\documentclass[12pt, reqno]{amsart}
\usepackage{amsmath, amsthm, amscd, amsfonts, amssymb, graphicx, color}
\usepackage[bookmarksnumbered, colorlinks, plainpages]{hyperref}
\hypersetup{colorlinks=true,linkcolor=red, anchorcolor=green, citecolor=cyan, urlcolor=red, filecolor=magenta, pdftoolbar=true}

\newcounter{vremennyj}

\newcommand{\ran}{\operatorname{ran}}

{\end{list}}

\numberwithin{equation}{section}

\newtheorem{thm}{Theorem}[section]
\newtheorem{lm}[thm]{Lemma}
\newtheorem{cor}[thm]{Corollary}
\newtheorem{prop}[thm]{Proposition}
\newtheorem{Definition}[thm]{Definition}
\newtheorem{ex}[thm]{Example}

\theoremstyle{remark}
\newtheorem{rem}[thm]{Remark}
\newtheorem*{rem*}{Remark}

\begin{document}

\setcounter{page}{1}

\title[$N$-hypercontractivity and similarity of Cowen-Douglas operators ]
{$N$-hypercontractivity and similarity of Cowen-Douglas operators  }

\author{Kui Ji} \author{Hyun-Kyoung Kwon}\author{Jing Xu}
\email{jikui@hebtu.edu.cn, hkwon6@albany.edu, 1746319384@qq.com}
\address{Department of Mathematics, Hebei Normal University, Shijiazhuang, Hebei, 050016, China}
\address{Department of Mathematics and Statistics, University at Albany- State University of New
York, Albany, NY, 12222, USA}
\thanks{
The first author is supported  by National Natural Science
Foundation of China (Grant No. 11831006)}

\subjclass[2000]{Primary 47C15; Secondary 47B37, 47B48, 47L40}

\keywords{Cowen-Douglas operator, similarity, eigenvector bundle, curvature, subharmonic function}

\begin{abstract}
When the backward shift operator on a weighted space $H^2_w=\{f=\sum_{j=0} ^{\infty} a_jz^j : \sum_{j=0}^{\infty} |a_j|^2w_j < \infty\}$ is an $n$-hypercontraction, we prove that the weights must satisfy the inequality
$$\frac{w_{j+1}}{w_j} \leq {\frac{1+j}{n+j}}.$$  As an application of this result, it is shown that such an operator cannot be subnormal. We also give an example to illustrate the important role that the $n$-hypercontractivity assumption plays in determining the similarity of Cowen-Douglas operators in terms of the curvatures of their eigenvector bundles.

\end{abstract}

\maketitle
\setcounter{tocdepth}{1}

\setcounter{section}{-1}

\section{Introduction}

\label{s1} 
In order to generalize the much-celebrated model theorem of B. Sz.-Nagy and C. Foias, J. Agler in \cite{Agler}, introduced the notion of an $n$-hypercontraction which extends that of a contraction. Let $\mathcal{H}$ be a separable Hilbet space and denote by $\mathcal{L}(\mathcal{H})$ the algebra of bounded, linear operators defined on $\mathcal{H}$. If $n$ is a positive integer, then an \emph{$n$-hypercontraction} is an operator $T \in \mathcal{L}(\mathcal{H})$ with
$$
\sum_{j=0}^k (-1)^j{k \choose j}(T^*)^jT^j \geq 0,
$$
for all $1 \leq k \leq n$.

For a real number $r$ and an integer $0 \leq k \leq r$, set
$$
{r\choose k} = \frac{r(r-1)\ldots(r-k+1)}{k!}.
$$
One then considers the Hilbert space $\mathcal{M}_n$ of analytic functions on the unit disk $\mathbb{D}$ defined as 
$$
\mathcal{M}_n:= \{ f= \sum_{k=0}^{\infty} \hat{f}(k)z^k: \sum_{k=0}^{\infty} |\hat{f}(k)|^2 \frac{1}{{n+k-1
\choose k}} < \infty \}.
$$
As can be easily checked, different function spaces correspond to different $n$'s: the Hardy space for $n=1$ and the weighted Bergman spaces $A^2_{n-2}$ for $n \geq 2$. The space $\mathcal{M}_n$ is a reproducing kernel Hilbert
space with kernel function given by
$$K_n(z,w)=\frac{1}{(1-\overline{w}z)^{n}},$$ for $ z, w \in \mathbb{D}$. The vector-valued spaces $\mathcal{M}_{n, \mathcal{E}}$ with values
in a separable Hilbert space $\mathcal{E}$ can also be naturally defined. The (forward) shift operator $S_{n,\mathcal{E}}$ on $\mathcal{M}_{n, \mathcal{E}}$ is defined as $$S_{n,\mathcal{E}}f(z):=zf(z),$$
and the  backward shift operator $S^*_{n,\mathcal{E}}$ is its
adjoint. 

We are now ready to state the following theorem by J. Agler:

\begin{thm}[\cite{Agler}, \cite{Agler2}]
\label{model}
For $T \in \mathcal{L}(\mathcal{H})$, there exist a Hilbert space $\mathcal{E}$ and an
$S^*_{n, \mathcal{E}}$-invariant subspace $\mathcal{N} \subseteq
\mathcal{M}_{n, \mathcal{E}}$ such that $T$ is unitarily
equivalent to $S^*_{n, \mathcal{E}}|\mathcal{N}$ if and only if
$T$ is an $n$-hypercontraction and $\lim\limits_{j\rightarrow \infty} \|T^j h\|=0$ for all
$h \in \mathcal{H}$.
\end{thm}

The functionality of the $n$-hypercontractivity assumption is also apparent in the study of similarity. By using Theorem \ref{model}, the second author, with R. G. Douglas and S. Treil, proved a similarity theorem between an $n$-hypercontractive Cowen-Douglas operator $T \in \mathcal{L}(\mathcal{H})$ and  the backward shift operator $S^*_n$ on $\mathcal{M}_n$ \cite{DKT}. Let us now recall the definition of a Cowen-Douglas operator.
\begin{Definition}[\cite{CD}] Let $\Omega$ be an open connected set of the complex plane
$\mathbb{C}$ and let $m$ be a positive integer. The Cowen-Douglas class $ B_{m}(\Omega)$ consists of operators $ T\in \mathcal{L}(\mathcal{H})$with the following conditions:
\begin{enumerate}
\item${\Omega}\text{ }{\subset}\text{ }{\sigma}(T)=\{w{\in} \mathbb{C} \text{}:
T-w {\mbox { is not invertible}}  \};$

\item $\ran (T-w)$ is closed $\text{ for every } w\in\Omega$;

\item $\bigvee \limits_{w{\in}{\Omega}} \ker(T-w)=\mathcal H$;
and

\item $\dim \ker (T-w)=m \text{ for every } w\in\Omega$.

\end{enumerate}

\end{Definition}

One of the main results of \cite{CD} states that each operator $T\in B_m(\Omega)$ induces a Hermitian holomorphic eigenvector bundle
$$\mathcal{E}_T:=\{(w, x)\in \Omega\times {\mathcal
H}: x \in \ker (T-w)\},$$ over $\Omega$. Since condition (4) implies that $\mathcal{E}_T$ is a bundle of rank $m$, we set $\{e_j(w)\}^{m}_{j=1}$ to be its holomorphic
frame. Letting
$$h(w):=(\langle e_j(w),e_i(w \rangle)_{m \times
m},$$
for each $w\in \Omega,$ the curvature function $\mathcal{K}_T$  of
$\mathcal{E}_T$ is defined as
$$\mathcal{K}_T=-\overline{\partial}(h^{-1}\partial h).$$
For $T\in B_1(\Omega)$, the curvature function is much simpler to calculate as it is equivalent to
$$
\mathcal{K}_T(w)=-\partial \bar{\partial} \log ||\gamma(w)||^2,
$$
where $\gamma(w)\in \ker(T-w)$ is a holomorphic cross section of $\mathcal{E}_T$ \cite{CD}.

More recently, the first two authors, along with Y. Hou, showed that the results of \cite{DKT} can be rephrased.
\begin{thm}[\cite{HJK}]
The following are equivalent:
\begin{enumerate}

\item An $n$-hypercontractive Cowen-Douglas operator $T \in B_m(\mathbb{D})$ is similar to $S^*_{n, \mathbb{C}^m}$ on $\mathcal{M}_{n,\mathbb{C}^m}$.

\item  $\partial \overline{\partial} \psi (w) \geq \text{trace }\mathcal{K}_{\bigoplus\limits^{m}_{j=1}S^*_{n}}(w)-\text{trace }\mathcal{K}_T(w),$
for some positive, bounded, subharmonic fiunction $\psi$ defined on $\mathbb{D}$ and for every $w \in \mathbb{D}$.

\end{enumerate}

\end{thm}

When $n=m=1$,  $T$ is a contraction and $S^*_1$ is just the adjoint of the shift operator on the Hardy space.
In \cite{KT1}, the second author and S. Treil gave an example of a backward shift operator $T$ (that is not a contraction) defined on a weighted space that is not similar to $S_1^*$ but such that it  still satisfies the inequality 
$$\partial \overline{\partial} \psi (w) \geq \mathcal{K}_{S_1^*}(w)- \mathcal{K}_T(w).$$
This means that one cannot ignore the contraction assumption when considering the similarity to the backwad shift operator on the Hardy space in terms of curvature. We try to do something analogous here and consider weighted spaces and $n$-hypercontractions. In particular, we give a necessary condition for the backward shift operator defined on a weighted space to be an $n$-hypercontraction. The first two cases are trivial to show. For $n \geq 3$, we make clever use of certain systems of linear equations with solutions that have negative entries. This work is done through the two lemmas in the next section. As corollaries of this result, we consider the subnormality problem of these weighted backward shift operators and also state a related result involving curvature. In the last section, we use $K$-theory to show that without the $n$-hypercontractivity assumption, the similarity criteria given in \cite{HJK} fails for the higher rank cases as well.
\section{$n$-hypercontractive backward shift operators}

The following theorem is our main result of the paper. 

\begin{thm}\label{main}
Let $T$ be the backward shift operator on the space $$H^2_w=\{f=\sum_{j=0} ^{\infty} a_jz^j : \sum_{j=0}^{\infty} |a_j|^2w_j < \infty\},$$ where $w_j >0$, $\liminf_j |w_j|^{\frac{1}{j}}=1$, and $\sup_j \frac{w_{j+1}}{w_j} < \infty$.
If $T$ is  an $n$-hypercontraction, then we have for every nonnegative integer $j$, $$\frac{w_{j+1}}{w_j}  \leq {\frac{1+j}{n+j}}.$$
\end{thm}

\

\begin{rem}
Note that the condition $\liminf_j |w_j|^{\frac{1}{j}}=1$ makes $H^2_w$ a space of analytic functions on the unit disk $\mathbb{D}$, while the condition $\sup_j \frac{w_{j+1}}{w_j} < \infty$ guarantees the boundedness of the shift operator on the space.
It is also easy to see that $T$ should be of the form
$$
T \left (\sum_{j=0} ^{\infty} a_jz^j \right )=\sum_{j=0} ^{\infty} \frac{w_{j+1}}{w_j} a_{j+1}z^j.
$$
Based on the definition given by A. L. Shields in \cite{SH}, $T$ is a weighted shift operator with weight sequence given by $\left \{ \sqrt {{\frac{w_{j+1}}{w_j}}} \right \}_{j=0} ^{\infty}$. 
\end{rem}
To give a proof of the above theorem, we need a few lemmas.

\begin{lm}\label{x1} Let $n \geq 2$ be a positive integer. For each $1 \leq j \leq k-1$, set
$$x_{j}=-{n+(j-2)\choose j}\frac{k-j}{k},$$
where $2 \leq k \leq n$. Then,

\

$\begin{bmatrix}\begin{smallmatrix}\setlength{\arraycolsep}{0.2pt}
\begin{array}{ccccccc}
-{n \choose 1}&1&0&\cdots&0&0\\
{n \choose 2}&-{n \choose 1}&1&\cdots&0&0\\
-{n \choose 3}&{n \choose 2}&-{n \choose 1}&\cdots&0&0\\
\vdots&\vdots&\vdots&\ddots&\vdots&\vdots\\
(-1)^{j}{n \choose j}&(-1)^{j-1}{n \choose j-1}&(-1)^{j-2}{n \choose j-2}&\cdots&0&0\\
\vdots&\vdots&\vdots&\ddots&\vdots&\vdots\\
(-1)^{k-2}{n \choose k-2} & (-1)^{k-3}{n \choose k-3} & (-1)^{k-4}{n \choose k-4} & \cdots & -{n \choose 1} & 1\\
(-1)^{k-1}{n \choose k-1} & (-1)^{k-2}{n \choose k-2} & (-1)^{k-3}{n \choose k-3} & \cdots & {n \choose 2} & -{n \choose 1}\\
\end{array}

\end{smallmatrix}\end{bmatrix}
\begin{bmatrix}
x_1 \\ x_2 \\ x_3 \\ \vdots \\ x_j \\ \vdots \\x_{k-2} \\ x_{k-1}
\end{bmatrix}
=
\begin{bmatrix}(-1)^{2}{n\choose 2}\\  (-1)^{3}{n\choose 3}\\ (-1)^{4}{n\choose 4}\\ \vdots \\ (-1)^{j+1}{n\choose j+1}\\  \vdots \\ (-1)^{k-1}{n\choose k-1}\\ (-1)^{k}{n\choose k}\end{bmatrix}$.

\

\end{lm}
\begin{proof}
The conclusion is equivalent to
$$\linespread{0.5}\selectfont
\left\{
\begin{array}{cc}
{n \choose 2}=-{n \choose 1}x_{1}+x_{2}\\

-{n \choose 3}={n \choose 2}x_{1}-{n \choose 1}x_{2}+x_{3}\\
\vdots\\
(-1)^{j+1}{n \choose j+1}=(-1)^{j}{n \choose j}x_{1}+(-1)^{j-1}{n \choose j-1}x_{2}+\cdots+(-1)^{1}{n \choose 1}x_{j}+x_{j+1}\\
\vdots\\
(-1)^{k-1}{n \choose k-1}=(-1)^{k-2}{n \choose k-2}x_{1}+(-1)^{k-3}{n \choose k-3}x_{2}+\cdots+(-1)^{1}{n \choose 1}x_{k-2}+x_{k-1}\\
\\
(-1)^{k}{n \choose k}=(-1)^{k-1}{n \choose k-1}x_{1}+(-1)^{k-2}{n \choose k-2}x_{2}+\cdots+(-1)^{2}{n \choose 2}x_{k-2}+(-1)^{1}{n \choose 1}x_{k-1}.
\end{array}
\right.
$$
We proceed by strong induction. Let $x_{1}=-{n+(1-2)\choose 1}\frac{k-1}{k}$
and substituting this into the first equation, we obtain $$x_{2}=-{n+(2-2)\choose 2}\frac{k-2}{k}.$$ For $2 \leq m \leq k-1$, we will set for each $1 \leq j \leq m-1$,
$$x_{j}=-{n+(j-2)\choose j}\frac{k-j}{k},$$
and prove that $$x_{m}=-{n+(m-2)\choose m}\frac{k-m}{k}.$$
This means that
$$x_{j}=-{n+(j-2)\choose j}\frac{k-j}{k},$$ for $1 \leq j \leq m$,
must satisfy the equations
\begin{equation}\label{one}
(-1)^{m}{n \choose m}-(-1)^{m-1}{n \choose m-1}x_{1}-(-1)^{m-2}{n \choose m-2}x_{2}-\cdots-(-1)^{1}{n \choose 1}x_{m-1}-x_{m}=0,
\end{equation}
and
\begin{equation}\label{two}
(-1)^{k}{n \choose k}=(-1)^{k-1}{n \choose k-1}x_{1}+(-1)^{k-2}{n \choose k-2}x_{2}+\cdots+(-1)^{2}{n \choose 2}x_{k-2}+(-1)^{1}{n \choose 1}x_{k-1}.
\end{equation}

Note that (\ref{one}) is equivalent to
\begin{eqnarray*}
0 &=&\sum\limits_{j=0}^{m}(-1)^{j}{n \choose j}{n+m-2-j\choose m-j}\frac{k-(m-j)}{k}\\
&=&\sum\limits_{j=0}^{m}(-1)^{j}{n \choose j}{n+m-2-j \choose m-j}-\frac{1}{k}\sum\limits_{j=0}^{m-1}(-1)^{j}(m-j){n \choose j}{n+m-2-j \choose m-j}.
\end{eqnarray*}
To show (\ref{one}), we will prove
\begin{equation}\label{first}\sum\limits_{j=0}^{m}(-1)^{j}{n \choose j}{n+m-2-j \choose m-j}=0,
\end{equation}
and \begin{equation} \label{second} \sum\limits_{j=0}^{m-1}(-1)^{j}(m-j){n \choose j}{n+m-2-j \choose m-j}=0.\end{equation}
Since
$$(1+x)^{n}=\sum\limits_{j=0}^{n}{n \choose j}x^{j},$$
and $$(1+x)^{-n}=\sum\limits_{j=0}^{\infty}{-n \choose j}x^{j}=\sum\limits_{j=0}^{\infty}(-1)^{j}{n+j-1 \choose j}x^{j},$$
for $n \geq 0$, it follows that
\begin{equation}\label{forfirst} 1+x=(1+x)^{n}\times(1+x)^{-(n-1)}= \left (\sum\limits_{j=0}^{n}{n \choose j}x^{j} \right ) \left (\sum\limits_{j=0}^{\infty}(-1)^{j}{n+j-2 \choose j}x^{j} \right ).\end{equation}

One then observes that the coefficients of $x^m$ for $m \geq 0$ on the right side of (\ref{forfirst}) are given by $$\sum\limits_{j=0}^{m}(-1)^{m-j}{n \choose j}{n+m-2-j \choose m-j}.$$ Comparing the coefficients of $x^m$ from both sides of (\ref{forfirst}) for $m \geq 2$ now yields (\ref{first}).

Similarly, since $$-\frac{n-1}{(1+x)^{n}}=[(1+x)^{-(n-1)}]^{\prime}= \left [ \sum\limits_{j=0}^{\infty}(-1)^{j}{n+j-2 \choose j}x^{j} \right ]'=\sum\limits_{j=1}^{\infty}(-1)^{j}j{n+j-2 \choose j}x^{j-1},$$
we have
\begin{equation}\label{forsecond}1-n=(1+x)^{n}\times\frac{1-n}{(1+x)^{n}}=\left (\sum\limits_{j=0}^{n}{n \choose j}x^{j}\right ) \left (\sum\limits_{j=1}^{\infty}(-1)^{j}j{n+j-2 \choose j}x^{j-1}\right ).\end{equation}

The coefficients of $x^{m-1}$ for $m \geq1$ on the right side of (\ref{forsecond}) are given by $$\sum \limits_{j=0}^{m-1}(-1)^{m-j}(m-j){n \choose j}{n+m-2-j \choose m-j},$$ and again, (\ref{second}) readily follows from comparing the coefficients in (\ref{forsecond}) for $m \geq 2$.

Lastly, since
\begin{eqnarray*}
0&=&(-1)^{k}{n \choose k}-(-1)^{k-1}{n \choose k-1}x_{1}-(-1)^{k-2}{n \choose k-2}x_{2}-\cdots-(-1)^{1}{n \choose 1}x_{k-1}\\
&=&\sum\limits_{j=1}^{k}(-1)^{j}{n \choose j}{n+k-2-j\choose k-j}\frac{k-(k-j)}{k}\\
&=&\sum\limits_{j=1}^{k}(-1)^{j}{n \choose j}{n+k-2-j\choose k-j}-\frac{1}{k}\sum\limits_{j=1}^{k-1}(-1)^{j}(k-j){n \choose j}{n+k-2-j \choose k-j}\\
&=&\sum\limits_{j=0}^{k}(-1)^{j}{n \choose j}{n+k-2-j \choose k-j}-\frac{1}{k}\sum\limits_{j=0}^{k-1}(-1)^{j}(k-j){n \choose j}{n+k-2-j \choose k-j},
\end{eqnarray*}
($\ref{two}$) holds based on what was done for ($\ref{one})$.
\end{proof}

\begin{lm}\label{x2}
Let $m, n \geq 2$ be positive integers. For each $1 \leq j \leq n+m-1$, set
$$x_{j}=-{n+(j-2)\choose j}\frac{n+m-j}{n+m}.$$
 Then,
$$\left [ \begin{smallmatrix}
(-1)^{1}{n \choose 1}&1&0&......&0&0&......&0&0&\\
(-1)^{2}{n \choose 2}&(-1)^{1}{n\choose 1}&1&......&0&0&......&0&0&\\
\vdots&\vdots&\vdots&\ddots&\vdots&\vdots&\ddots&\vdots&\vdots&\\
(-1)^{j}{n\choose j}&(-1)^{j-1}{n\choose j-1}&(-1)^{j-2}{n\choose j-2}&\cdots&0&0&......&0&0&\\
\vdots&\vdots&\vdots&\ddots&\vdots&\vdots&\ddots&\vdots&\vdots&\\
(-1)^{n-1}{n\choose n-1}&(-1)^{n-2}{n\choose n-2}&(-1)^{n-3}{n\choose n-3}&\cdots&0&0&......&0&0&\\
(-1)^{n}{n\choose n}&(-1)^{n-1}{n\choose n-1}&(-1)^{n-2}{n\choose n-2}&\cdots&1&0&......&0&0&\\
0&(-1)^{n}{n\choose n}&(-1)^{n-1}{n\choose n-1}&\cdots&(-1)^{1}{n\choose 1}&1&\cdots&0&0&\\
\vdots&\vdots&\vdots&\ddots&\vdots&\vdots&\ddots&\vdots&\vdots&\\
0&0&0&\cdots&0&0&\cdots&(-1)^{2}{n\choose 2}&(-1)^{1}{n\choose 1}
\end{smallmatrix}\right ]
\left [\begin{smallmatrix}
x_{1}\\
x_{2}\\
\vdots\\
x_{j}\\
\vdots\\
x_{n-1}\\
x_{n}\\
x_{n+1}\\
\vdots\\
x_{n+m-1}\\
\end{smallmatrix}\right ]
=
\left [ \begin{smallmatrix}
(-1)^{2}{n\choose 2}\\
(-1)^{3}{n\choose 3}\\
\vdots\\
(-1)^{j+1}{n\choose j+1}\\
\vdots\\
(-1)^{n}{n\choose n}\\
0\\
0\\
\vdots\\
0\\
\end{smallmatrix}\right ].$$

\end{lm}

\begin{proof}
The conclusion is equivalent to 
$${\linespread{0.5}\selectfont
\left\{
\begin{array}{cc}
{n \choose 2}=-{n \choose 1}x_{1}+x_{2}\\
\\
-{n \choose 3}={n \choose 2}x_{1}-{n \choose 1}x_{2}+x_{3}\\
\\
\vdots\\
\\
(-1)^{j+1}{n \choose j+1}=(-1)^{j}{n \choose j}x_{1}+(-1)^{j-1}{n \choose j-1}x_{2}+\cdots+(-1)^{1}{n \choose 1}x_{j}+x_{j+1}\\
\\
\vdots\\
\\
(-1)^{n}{n \choose n}=(-1)^{n-1}{n \choose n-1}x_{1}+(-1)^{n-2}{n \choose n-2}x_{2}+\cdots+(-1)^{1}{n \choose 1}x_{n-1}+x_{n}\\
\\
0=(-1)^{n}{n \choose n}x_1+(-1)^{n-1}{n \choose {n-1}}x_2+\cdots +(-1)^{1}{n \choose 1}x_{n}+x_{1+n}
\\
\vdots\\

\\
0=(-1)^{n}{n \choose n}x_{i}+(-1)^{n-1}{n \choose n-1}x_{i+1}+\cdots+(-1)^{1}{n \choose 1}x_{i+n-1}+x_{i+n}\\
\\
\vdots\\
\\
0=(-1)^{n}{n \choose n}x_{m}+(-1)^{n-1}{n \choose n-1}x_{m+1}+\cdots+(-1)^{1}{n \choose 1}x_{n+m-1}.
\end{array}
\right.}
$$

If we set $$x_{1}=-{n+(1-2) \choose 1}\frac{n+m-1}{n+m},$$
then according to Lemma \ref{x1} with $k=n$, $$x_{j}=-{n+(j-2)\choose j}\frac{n+m-j}{n+m},$$ for all $1 \leq j \leq n-1$.

As in the previous lemma, we then show that  $$x_{j}=-{n+(j-2)\choose j}\frac{n+m-j}{n+m},$$ defined for $1 \leq  j \leq n+m-1$, satisfy the following three equations:
First,
\begin{equation}\label{secondfirst}
0=(-1)^{n}{n \choose n}-(-1)^{n-1}{n \choose n-1}x_{1}-(-1)^{n-2}{n \choose n-2}x_{2}-\cdots-(-1)^{1}{n \choose 1}x_{n-1}-x_{n}.
\end{equation}
Second, for $1 \leq i \leq m-1$,
\begin{equation}\label{secondsecond}
0=(-1)^{n}{n \choose n}x_{i}+(-1)^{n-1}{n \choose n-1}x_{i+1}+\cdots+(-1)^{1}{n \choose 1}x_{i+n-1}+x_{i+n},
\end{equation}
and third,
\begin{equation}\label{secondthird}
0=(-1)^{n}{n \choose n}x_{m}+(-1)^{n-1}{n \choose n-1}x_{m+1}+\cdots+(-1)^{2}{n \choose 2}x_{n+m-2}+(-1)^{1}{n \choose 1}x_{n+m-1}.
\end{equation}

For (\ref{secondfirst}), we see that it is equivalent to
\begin{eqnarray*}
0 &=&\sum\limits_{j=0}^{n}(-1)^{j}{n \choose j}{2n-2-j \choose n-j}\frac{n+m-(n-j)}{n+m}\\
&=&\sum\limits_{j=0}^{n}(-1)^{j}{n \choose j}{2n-2-j \choose n-j}-\frac{1}{n+m}\sum\limits_{j=0}^{n}(-1)^{j}(n-j){n \choose j}{2n-2-j \choose n-j},
\end{eqnarray*}
while to prove (\ref{secondsecond}), we show that for $1 \leq i \leq m-1$,

\begin{eqnarray*}
0 &=&-\sum\limits_{j=0}^{n}(-1)^{j}{n \choose j}{2n+i-2-j \choose n+i-j}\frac{n+m-(n+i-j)}{n+m}\\
&=&-\sum\limits_{j=0}^{n}(-1)^{j}{n \choose j}{2n+i-2-j \choose n+i-j}+\frac{1}{n+m}\sum\limits_{j=0}^{n}(-1)^{j}(n+i-j){n \choose j}{2n+i-2-j \choose n+i-j}.
\end{eqnarray*}

Finally, (\ref{secondthird}) amounts to showing
\begin{eqnarray*}
0 &=&-\sum\limits_{j=1}^{n}(-1)^{j}{n \choose j}{2n+m-2-j \choose n+m-j}\frac{n+m-(n+m-j)}{n+m}\\
&=&-\sum\limits_{j=1}^{n}(-1)^{j}{n \choose j}{2n+m-2-j \choose n+m-j}+\frac{1}{n+m}\sum\limits_{j=1}^{n}(-1)^{j}(n+m-j){n \choose j}{2n+m-2-j \choose n+m-j}\\
&=&-\sum\limits_{j=0}^{n}(-1)^{j}{n \choose j}{2n+m-2-j \choose n+m-j}+\frac{1}{n+m}\sum\limits_{j=0}^{n}(-1)^{j}(n+m-j){n \choose j}{2n+m-2-j \choose n+m-j}.
\end{eqnarray*}

Note that once we prove that
\begin{equation} \label{endfirst}
\sum\limits_{j=0}^{n}(-1)^{j}{n \choose j}{n+k-2-j \choose k-j}=0,
\end{equation}
and that
\begin{equation} \label{endsecond}
\sum\limits_{j=0}^{n}(-1)^{j}(k-j){n \choose j}{n+k-2-j \choose k-j}=0,
\end{equation}
for $n \leq k \leq n+m$, the
equations (\ref{secondfirst}), (\ref{secondsecond}), and (\ref{secondthird}) will immediately follow.

But it was already calculated in the previous lemma that for each $k \geq 0$, the term 
$$
\sum\limits_{j=0}^{k}(-1)^{k-j}{n \choose j}{n+k-2-j \choose k-j}
$$
represents the coefficient of $x^k$ in the expression $1+x$. Since $k \geq n \geq 2$ and ${n \choose k}=0$ for $k > n$, (\ref{endfirst}) holds. One can show analogously that (\ref{endsecond}) is true by using the fact that
 $$\sum \limits_{j=0}^{k-1}(-1)^{k-j}(k-j){n \choose j}{n+k-2-j \choose k-j}$$
 is the coefficient of $x^{k-1}$ in the expression $1-n$ for $k \geq 1$.

\end{proof}

\subsection{Proof of Theorem \ref{main}}

\begin{proof}
The operator $T$ is of the form
$$
T \left (\sum_{j=0} ^{\infty} a_jz^j \right )=\sum_{j=0} ^{\infty} \frac{w_{j+1}}{w_j} a_{j+1}z^j, 
$$
and for the sake of simplicity, we will now set $$\lambda_j:=\frac{w_{j+1}}{w_j}.$$
Then,
$$
T=\begin{pmatrix} 0 & \sqrt{\lambda_0} & 0 & 0 & 0 & \cdots \\ 0 & 0 & \sqrt{\lambda_1} & 0 & 0 & \cdots \\ 0 & 0 & 0 & \sqrt{\lambda_2} & 0 & \cdots \\ \vdots  & \vdots & \vdots  & & \ddots\end{pmatrix},
$$
and for every $m \geq 1$, $T^{*m}T^m$ is the diagonal matrix with the nonzero entry $$\prod_{j=k-1}^{m+k-2} \lambda_j,$$ in the $(m+k) \times (m+k)$ position for each positive integer $k$.

If $T$ is a $1$-hypercontraction, that is,
 $$I-T^*T \geq 0,$$
 then by looking at the entries of $I-T^*T$, we have
 $$
 1-\lambda_j \geq 0,
 $$
 for every nonnegative integer $j$.
 This means that $\lambda_j \leq1={\frac{1+j}{1+j}}$ for every nonnegative integer $j$.\\

 If $T$ is a $2$-hypercontraction so that $$
\sum_{j=0}^2 (-1)^j{2 \choose j}(T^*)^jT^j \geq 0,
$$ then $$\begin{cases}1\geq0 \\ 1-{2 \choose 1}\lambda_0 \geq0\\
1-{2 \choose 1}\lambda_1+{2 \choose 2}\lambda_1 \lambda_0 \geq0\\
\vdots\\
1-{2 \choose 1}\lambda_k+{2 \choose 2} \ \lambda_k \lambda_{k-1}\geq0	\\
\vdots\\
\end{cases}.$$
From this, it is easy to see that $$\lambda_0 \leq {\frac{1+0}{2+0}} \text{ and }\lambda_1 \leq\frac{1}{2-\lambda_0}\ \leq \frac{1+1}{2+1}.$$

Now if we suppose that $$\lambda_{k-1} \leq {\frac{1+(k-1)}{2+(k-1)}},$$ then it follows that $$\lambda_k \leq \frac{1}{2-\lambda_{k-1}} \leq \frac{1+k}{2+k}.$$\\

If $T$ is  an $n$-hypercontraction for $n\geq3$, then $$
\sum_{j=0}^n (-1)^j{n \choose j}(T^*)^jT^j \geq 0,$$
which equals
\begin{equation}
\begin{cases}
1\geq0\qquad\qquad\qquad\qquad\qquad\qquad\qquad\qquad\qquad\qquad\qquad\qquad\qquad \\
1-{n \choose 1}\lambda_0 \geq 0 \qquad\qquad\qquad\,\,\,\,\,\,\,\qquad\qquad\qquad\,\qquad\qquad\qquad\qquad (1)\\
1-{n \choose 1}\lambda_1 +{n \choose 2}\lambda_1 \lambda_0 \geq0\qquad\qquad\qquad\,\,\,\,\,\,\,\,\,\,\,\,\,\,\,\,\,\,\,\,\,\,\, \qquad\qquad\qquad\,\,\,\,\, (2)\\
1-{n \choose 1}\lambda_2+{n \choose 2}\lambda_2 \lambda_1-{n \choose 3}\lambda_2 \lambda_1 \lambda_0 \geq0\, \, \, \,\,\,\, \, \, \, \, \, \, \, \, \, \, \,\,   \qquad \qquad \qquad \qquad (3)\\
\vdots\\
1+\sum\limits_{j=1}^{k}(-1)^{j}{n \choose j}\lambda_{k-1}\lambda_{k-2}\cdots \lambda_{k-j} \geq0\,\,\,\,\,\,\,\qquad \qquad \qquad \qquad\qquad\,(k)\\
\vdots\\
1+\sum\limits_{j=1}^{n}(-1)^{j}{n \choose j}\lambda_{n-1}\lambda_{n-2}\cdots \lambda_{n-j} \geq0\,\,\,\,\,\,\,\qquad \qquad \qquad \qquad\qquad(n)\\

1+\sum\limits_{j=1}^{n}(-1)^{j}{n \choose j}\lambda_{n}\lambda_{n-1}\cdots \lambda_{n+1-j} \geq0\,\,\,\,\qquad\qquad \qquad \qquad\qquad(n+1)\\
\vdots\\
1+\sum\limits_{j=1}^{n}(-1)^{j}{n \choose j}\lambda_{n+m-1}\lambda_{n+m-2}\cdots \lambda_{n+m-j} \geq0 \,\,\,\,\,\,\,\qquad \qquad \qquad(n+m)\\
\vdots
\end{cases}
\end{equation}
From inequality $(1)$, we get $$\lambda_0 \leq\frac{1}{n}=\frac{1+0}{n+0},$$
and we use it together with inequality $(2)$ to obtain $$\lambda_1 \leq\frac{1+1}{n+1}.$$

It is now the right time to resort to the lemmas that have been proved previously. Namely, by Lemma $1.2$, the $x_j=-{n+(j-2) \choose j}\frac{3-j}{3}$, for $1 \leq j \leq 2$, satisfy the equation
$$
\begin{bmatrix}
-{n \choose 1}&1\\
{n \choose 2}&-{n \choose 1}\\
\end{bmatrix}
\begin{bmatrix}
\begin{array}{cc}
 x_{1}\\
 x_{2}\\
\end{array}
\end{bmatrix}
=
\begin{bmatrix}
\begin{array}{cc}
{n \choose 2}\\
-{n \choose 3}\\
\end{array}
\end{bmatrix},
$$
that is,
$$
\left\{
\begin{array}{cc}
{n \choose 2}=-{n \choose 1}x_{1}+x_{2} \\
\\
-{n \choose 3}={n \choose 2}x_{1}-{n \choose 1}x_{2}
\end{array}.
\right.
$$
Plugging this into inequality $(3)$, we have

\

$0 \leq 1-{n \choose 1}\lambda_2+{n \choose 2}\lambda_2 \lambda_1-{n \choose 3}\lambda_2\lambda_1\lambda_0\\$
\\
$=1-{n \choose 1}\lambda_2+[-{n \choose 1}x_1+x_2 ]\lambda_2 \lambda_1+[{n \choose 2}x_1-{n \choose 1}x_2 ]\lambda_2 \lambda_1 \lambda_0\\$
\\
$=1-\lambda_2[{n \choose 1}+x_{1}]+x_{1}\lambda_2[1-{n \choose 1}\lambda_1+{n \choose 2}\lambda_1 \lambda_0 ]+x_{2}\lambda_2 \lambda_1 [1-{n \choose 1}\lambda_0].$

\

From the  inequalities $(1)$ and $(2)$ of $(1.12)$,
we obtain
$$1-\left [{n \choose 1}+x_{1}\right  ]\lambda_2 \geq 0,$$
by taking into account that for $1 \leq j \leq 2$,
$$
x_j <0.
$$
Thus, $$\lambda_2 \leq\frac{1}{{n \choose 1}+x_{1}}=\frac{1+2}{n+2}.$$
In general, recall that Lemma \ref{x1} states that for $2 \leq k \leq n$ and $x_{j}=-{n+(j-2) \choose j}\frac{k-j}{k}<0$, where $1 \leq j \leq k-1$, we have

$$\linespread{0.5}\selectfont
\left\{
\begin{array}{cc}
{n \choose 2}=-{n \choose 1}x_{1}+x_{2} \\
\\
-{n \choose 3}={n \choose 2}x_{1}-{n \choose 1}x_{2}+x_{3}\\
\\
\vdots\\
\\
(-1)^{j}{n \choose j}=(-1)^{j-1}{n \choose j-1}x_{1}+(-1)^{j-2}{n \choose j-2}x_{2}+\cdots+(-1)^{1}{n \choose 1}x_{j-1}+x_{j}\\
\\
\vdots\\
\\
(-1)^{k-1}{n \choose k-1}=(-1)^{k-2}{n \choose k-2}x_{1}+(-1)^{k-3}{n \choose k-3}x_{2}+\cdots+(-1)^{1}{n \choose 1}x_{k-2}+x_{k-1}\\
\\
(-1)^{k}{n \choose k}=(-1)^{k-1}{n \choose k-1}x_{1}+(-1)^{k-2}{n \choose k-2}x_{2}+\cdots+(-1)^{2}{n \choose 2}x_{k-2}+(-1)^{1}{n \choose 1}x_{k-1}\\
\end{array}.
\right.
$$
Then from inequality $(k)$ of (1.12), we have
{\footnotesize{\begin{eqnarray*}
&&1+\sum\limits_{j=1}^{k}(-1)^{j}{n \choose j}\lambda_{k-1}\lambda_{k-2}\cdots\lambda_{k-j}\\
&=&1-{n \choose 1}\lambda_{k-1}+
\sum\limits_{j=2}^{k-1}\left[\sum\limits_{l=0}^{j-1}(-1)^{l}{n \choose l}x_{j-l} \right]\lambda_{k-1}\lambda_{k-2}\cdots \lambda_{k-j}
+\left[\sum\limits_{j=1}^{k-1}(-1)^{j}{n \choose j}x_{k-j}\right]\prod\limits_{l=1}^{k}\lambda_{k-l}\\
&=&1-\left[{n \choose 1}+x_{1}\right]\lambda_{k-1}
+\sum\limits_{j=1}^{k-1}
\left[1+\sum\limits_{l=1}^{k-j}(-1)^{l}{n \choose l}\lambda_{k-j-1}\lambda_{k-j-2}\cdots\lambda_{k-j-l}\right]x_{j}\lambda_{k-1}\lambda_{k-2}\cdots\lambda_{k-j}\\
&\geq&0.
\end{eqnarray*}}}
Now based on the inequalities $(1)$ through $(k)$ of $(1.12)$, we have for every $1 \leq m \leq k$, $$1+\sum\limits_{j=1}^{m}(-1)^{j}{n \choose j}\lambda_{m-1}\lambda_{m-2}\cdots\lambda_{m-j}\geq0,$$
and therefore, using the fact that $x_j < 0$ for every $1 \leq j \leq k-1$, the inequality
$$1-\left[{n \choose 1}+x_{1}\right]\lambda_{k-1}\geq0,$$
follows. Then for every $2 \leq k \leq n$, $$\lambda_{k-1}\leq\frac{1}{{n \choose 1}+x_{1}}=\frac{1+(k-1)}{n+(k-1)}.$$ Since it has been observed already that $\lambda_0 \leq \frac{1+0}{n+0}$, the inequality holds for all $1 \leq k \leq n$.

\

For $n+1 \leq k$, we make use of Lemma 1.3 that states that for $$x_{j}=-{n+(j-2) \choose j}\frac{n+m-j}{n+m},$$ with $1 \leq j \leq n+m-1$ and $n, m \geq 2$, one has
$$
{\linespread{0.5}\selectfont
\left\{
\begin{array}{cc}
{n \choose 2}=-{n \choose 1}x_{1}+x_{2}\\
\\
-{n \choose 3}={n \choose 2}x_{1}-{n \choose 1}x_{2}+x_{3}\\
\\
\vdots\\
\\
(-1)^{j+1}{n \choose j+1}=(-1)^{j}{n \choose j}x_{1}+(-1)^{j-1}{n \choose j-1}x_{2}+\cdots+(-1)^{1}{n \choose 1}x_{j}+x_{j+1}\\
\\
\vdots\\
\\
(-1)^{n}{n \choose n}=(-1)^{n-1}{n \choose n-1}x_{1}+(-1)^{n-2}{n \choose n-2}x_{2}+\cdots+(-1)^{1}{n \choose 1}x_{n-1}+x_{n}\\
\\
0=(-1)^{n}{n \choose n}x_1+(-1)^{n-1}{n \choose {n-1}}x_2+\cdots +(-1)^{1}{n \choose 1}x_{n}+x_{1+n}
\\
\vdots\\

\\
0=(-1)^{n}{n \choose n}x_{i}+(-1)^{n-1}{n \choose n-1}x_{i+1}+\cdots+(-1)^{1}{n \choose 1}x_{i+n-1}+x_{i+n}\\
\\
\vdots\\
\\
0=(-1)^{n}{n \choose n}x_{m}+(-1)^{n-1}{n \choose n-1}x_{m+1}+\cdots+(-1)^{1}{n \choose 1}x_{n+m-1}
\end{array}.
\right.}
$$

\

Now, by inequality $(n+m)$ in $(1.12),$ we have

$\footnotesize{\begin{array}{llll} \,&1+\sum\limits_{j=1}^{n}(-1)^{j}{n \choose j}\lambda_{n+m-1}\lambda_{n+m-2}\cdots\lambda_{n+m-j}\\
=&1-{n \choose 1}\lambda_{n+m-1}+
\sum\limits_{j=2}^{n}\left [\sum\limits_{l=0}^{j-1}(-1)^{l}{n \choose l}x_{j-l}\right ]\lambda_{n+m-1}\lambda_{n+m-2}\cdots \lambda_{n+m-j}\\
\,&+\sum\limits_{j=1}^{m-1}\left [\sum\limits_{l=0}^{n}(-1)^{l}{n \choose l}x_{n+j-l}\right ]\lambda_{n+m-1}\lambda_{n+m-2}\cdots\lambda_{m-j}+
\left[\sum\limits_{j=1}^{n}(-1)^{j}{n \choose j}x_{n+m-j}\right]\prod\limits_{l=1}^{n+m}\lambda_{l-1}\\
=&1-\left[{n \choose 1}+x_{1}\right]\lambda_{n+m-1}\\
&+\, \sum\limits_{j=1}^{m}
\left(\left[1+\sum\limits_{l=1}^{n}(-1)^{l}{n \choose l}\lambda_{n+m-j-1}\lambda_{n+m-j-2}\cdots \lambda_{n+m-j-l}\right]x_{j}\lambda_{n+m-1}
\cdots\lambda_{n+m-j}\right)\\
&+
\sum\limits_{j=m+1}^{n+m-1}\left(
\left[1+\sum\limits_{l=1}^{n+m-j}(-1)^{l}{n \choose l}\lambda_{n+m-j-1}\cdots\lambda_{n+m-j-l}\right]x_{j}\lambda_{n+m-1}
\cdots\lambda_{n+m-j}\right)\\
\geq&0.
\end{array}}$

\

Again, the inequalities in $(1.12)$ give for all $1 \leq k \leq n$,
$$
1+\sum\limits_{j=1}^{k}(-1)^{j}{n \choose j}\lambda_{k-1}\lambda_{k-2}\cdots\lambda_{k-j}\geq0,
$$
and for all $l \geq 0$,
$$
1+\sum\limits_{j=1}^{n}(-1)^{j}{n \choose j}\lambda_{n+l-1}\lambda_{n+l-2}\cdots\lambda_{n+l-j}\geq0.
$$
Since $$x_{j}=-{n+(j-2) \choose j}\frac{n+m-j}{n+m}<0,$$ for $1 \leq j \leq n+m-1$,
we conclude that
$$1-\left[{n \choose 1}+x_{1} \right]\lambda_{n+m-1} \geq 0.$$

Thus, $$\lambda_{n+m-1}\leq\frac{1}{{n \choose 1}+x_{1}}=\frac{1+(n+m-1)}{n+(n+m-1)},$$ for every $m \geq 2$, and we then have
 $$\lambda_j =\frac{w_{j+1}}{w_j} \leq \frac{1+j}{n+j},$$ for every nonnegative integer $j$.
\end{proof}

Theorem \ref{main} readily yields the following results. Recall that a \emph{subnormal} operator is an operator with a normal extension.

\begin{cor}
A weighted backward shift operator cannot be subnormal.
\end{cor}

\begin{proof} It is known that an operator is an $n$-hypercontraction for all $n$ if and only if it is a subnormal contraction (\cite{Agler2}).  Let $T$ be the backward shift operator on one of the spaces $H^2_w$ with weight sequence $\left \{\sqrt{\frac{w_{j+1}}{w_j}} \right \}_{j=0}^{\infty}$ and let it be subnormal. Since subnormality is preserved under the scalar multiplication operation, we can assume without generality that $\|T\| \leq 1$. Then by Theorem \ref{main}, for any fixed integer $j \geq 0$, we have for every integer $n \geq 1$, $\frac{w_{j+1}}{w_j} \leq \frac{1+j}{n+j}$. Since $\lim\limits_{n\rightarrow \infty} {\frac{1+j}{n+j}}=0,$ for every integer $j \geq 0$, $$\limsup_j \frac{w_{j+1}}{w_j}=0,$$ which is a contradiction to
 $\liminf_j |w_j|^{\frac{1}{j}}=1$.

\end{proof}

Next, let us recall how given an integer $n \geq 1$, the Hilbert space $\mathcal{M}_n$ of functions on the unit disk $\mathbb{D}$ is defined:
$$
\mathcal{M}_n= \{ f= \sum_{k=0}^{\infty} \hat{f}(k)z^k: \sum_{k=0}^{\infty} |\hat{f}(k)|^2 \frac{1}{{n+k-1
\choose k}} < \infty \}.
$$

Using the proof of Theorem \ref{main}, one can also show that the backward shift operator $S^*_n$ on $\mathcal{M}_n$ is ``almost" $n$-isometric''.

\begin{cor} Set $\mathcal{M}^1_n=\bigvee \{z^m: m \geq 1\}\subset \mathcal{M}_n$ and denote by $\mathcal{P}_n^1$ the orthogonal projection from
$\mathcal{M}_n $ to $\mathcal{M}^1_n$. Then
$$\mathcal{P}_n^1 \left (\sum_{j=0}^n (-1)^j{n \choose j}(S_n)^j(S^*_n)^j \right )\bigg \vert_{\mathcal{M}^1_n}=0.$$
\end{cor}

In addition, we construct in the next corollary a weighted space whose backward shift operator satisfies an inequality involving curvatures with respect to the operator $S^*_n$ on $\mathcal{M}_n$. This inequality looks almost the same as the one that appears in the similarity criteria but one can no longer say anything about subharmonicity. The following well-known result by A. L. Shields that helps determine when two weighted shift operators are similar will be used in one part of the proof.

\begin{lm}[\cite{SH}]\label{SH}

Let $T_1$ and $T_2$ be unilateral shifts with weight sequences $\{\lambda_j\}_{j=0}^{\infty}$ and $\{\tilde{\lambda}_j\}_{j=0}^{\infty}$, respectively. Then $T_1$ and $T_2$ are similar if and only if there exist positive constants $C_1$ and $C_2$ such that
$$0<C_1\leq \Big |\frac{\lambda_k\cdots \lambda_j}{\tilde{\lambda}_k\cdots \tilde{\lambda}_j} \Big | \leq C_2,$$
for all $k\leq j$.

\end{lm}

\begin{cor}\label{maincor} For each operator $S^*_n$ on $\mathcal{M}_n$, there exist a weighted backward shift operator $T$ that is not an $n$-hypercontraction and a positive, bounded, real-analytic function $\psi$ defined on the unit disk $\mathbb{D}$ such that $$\partial\bar{\partial}  \psi(w)=\mathcal{K}_{S^*_n}(w)-\mathcal{K}_T(w),$$ for every $w \in \mathbb{D}$. Moreover, $T$ is not similar to $S^*_n$.
\end{cor}
\begin{proof} We will define our backward shift operator $T$ on some weighted space $$\mathcal{H}=\left \{ \sum_{j=0}^{\infty} a_jz^j: \sum_{j=0}^{\infty} |a_j|^2w_j < \infty \right \}.$$  The operator $S^*_n$ is  an $n$-hypercontraction and
using the reproducing kernel for the space $\mathcal{M}_n$, we have $$\mathcal{K}_{S^*_n}(w)=-\partial\overline{\partial}\log\frac{1}{(1-|w|^{2})^{n}}.$$
If we write $\mathcal{K}_T$ as $$\mathcal{K}_{T}(w)=-\partial\overline{\partial}\log k_{{w}}(w),$$ where $k_{{w}}(z)=\sum\limits_{j=0}^{\infty}\frac{\overline{w}^jz^j}{w_j}$ denotes the reproducing kernel of $\mathcal{H}$, then
$$\mathcal{K}_{S^*_n}(w)-\mathcal{K}_{T}(w)=\partial\overline{\partial}\log\left [k_{{w}}(w){(1-|w|^{2})^{n}}\right ].$$ Hence, in order to prove that a positive, bounded, real-analytic function $\psi$ exists, we have to show that $k_{{w}}(w){(1-|w|^{2})^{n}}$ is bounded above and below by positive constants.


We first consider the sequence  $${\widetilde w}_{j}:=\frac{j!(n-1)!}{(n+j-1)!},$$
that appears in the following familiar expansion for $(1-x)^{-n}$:
$$
(1-x)^{-n}=\sum_{j=0}^{\infty}{ {n+j-1}\choose{j}} x^j.
$$

We now construct the sequence ${w}_j$ for the space $\mathcal{H}$. Let
$$w_{j}=
\begin{cases}
l{\widetilde w}_{j},\qquad j=N_{i}+l, \text{ for }1 \leq l \leq i,\\
l{\widetilde w}_{j},\qquad j=N_{i}+2i-l, \text{ for }1 \leq l \leq i,\\
{\widetilde w}_j,\qquad\qquad otherwise,
\end{cases}$$
where the sequence $\{N_i\}_{i \geq 1}$ consists of positive integers $N_i  > n-2$ with $$N_i+2i < N_{i+1}.$$ More details on the $N_i$ will be given later. Then, since $\frac{1}{w_{j}}-\frac{(n+j-1)!}{j!(n-1)} \neq 0$ only for $j=N_i+p$, where $2 \leq p \leq 2i-2$, we have

\begin{eqnarray*}k_{w}(w)&=&\sum\limits_{j=0}^{\infty}\frac{1}{w_{j}}(|w|^{2})^{j}\\
&=&\frac{1}{(1-|w|^{2})^{n}}+\sum\limits_{j=2}^{\infty}\left[\frac{1}{w_{j}}-\frac{(n+j-1)!}{j!(n-1)!}\right ](|w|^{2})^{j}\\
&=&\frac{1}{(1-|w|^{2})^{n}}+\sum\limits_{i=2}^{\infty} \sum\limits_{j=N_i+2}^{N_i+2i-2}\left[\frac{1}{w_{j}}-\frac{(n+j-1)!}{j!(n-1)!}\right ](|w|^{2})^{j}\\
&=:&\frac{1}{(1-|w|^2)^n}+\sum\limits_{i=2}^{\infty} g_i(w),
\end{eqnarray*}
where,
 \begin{eqnarray*}
|g_{i}(w)|=& \left | \sum\limits_{j=2}^{i}\left[\frac{1}{w_{N_{i}+j}}-\frac{(n+N_{i}+j-1)!}{(N_{i}+j)!(n-1)!}\right]|w|^{2N_{i}+2j}+\sum\limits_{j=2}^{i-1}\left[\frac{1}{w_{N_{i}+2i-j}}-\frac{(n+N_{i}+2i-j-1)!}{(N_{i}+2i-j)!(n-1)!}\right]|w|^{2N_{i}+4i-2j} \right | \\
=&\frac{(n+N_{i}-1)!}{N_{i}!(n-1)!}|w|^{2N_{i}}\left | \sum\limits_{j=2}^{i}\frac{(n+N_{i}+j-1)!N_{i}!}{(n+N_{i}-1)!(N_{i}+j)!}\left (\frac{1}{j}-1 \right)|w|^{2j}+\sum\limits_{j=2}^{i-1}\frac{(n+N_{i}+2i-j-1)!N_{i}!}{(n+N_{i}-1)!(N_{i}+2i-j)!}\left (\frac{1}{j}-1 \right)|w|^{4i-2j}\right |\\
\leq&\frac{(n+N_{i}-1)!}{N_{i}!(n-1)!}|w|^{2N_{i}}\left |\sum\limits_{j=2}^{i}2^{i}\left(\frac{1}{j}-1\right)|w|^{2j}+\sum\limits_{j=2}^{i-1}2^{i}\left (\frac{1}{j}-1\right)|w|^{4i-2j}\right |.
\end{eqnarray*}
Next, set
$$
M_i:=\sup  \limits_{|w|<1} \left |\sum\limits_{j=2}^{i}2^{i}\left(\frac{1}{j}-1\right)|w|^{2j}+\sum\limits_{j=2}^{i-1}2^{i}\left (\frac{1}{j}-1\right)|w|^{4i-2j}\right |,
$$
a constant that depends only on $i$ and not on the $N_i$. By direct calculation, one easily sees that  $$M_i<i2^{i+1}.$$
We then have $$k_{w}(w)(1-|w|)^{n}=1+\sum\limits_{i=2}^{\infty}g_{i}(w)(1-|w|^{2})^{n},$$
and $$|g_{i}(w)|(1-|w|^{2})^{n}\leq M_{i}\frac{(n+N_{i}-1)!}{N_{i}!(n-1)!}(|w|^{2})^{N_{i}}(1-|w|^{2})^{n}.$$

Now, for $x \in \mathbb{D}$, if we let $$f(x):=x^{N_{i}}(1-x)^{n},$$
then $$f'(x)=N_{i}x^{N_{i}-1}(1-x)^{n}-nx^{N_{i}}(1-x)^{n-1}=x^{N_{i}-1}(1-x)^{n-1}\left[N_{i}(1-x)-nx \right].$$
Since the function $f(x)$ attains a maximum of $\left (\frac{N_{i}}{N_{i}+n} \right)^{N_{i}}\left (\frac{n}{N_{i}+n} \right)^{n}$ at $x=\frac{N_{i}}{N_{i}+n},$
$$|g_{i}(w)|(1-|w|^{2})^{n}\leq M_{i}\frac{(n+N_{i}-1)!}{N_{i}!(n-1)!}\left (\frac{N_{i}}{N_{i}+n} \right )^{N_{i}}\left (\frac{n}{N_{i}+n} \right )^{n}
\leq M_{i}\frac{n^{n}}{(n-1)!(N_{i}+n)}.$$

Now if we choose
 $$N_{i}>max \left [\frac{2^{i+2}M_{i}n^{n}}{(n-1)!}-n, n-2 \right],$$
 then $$|g_{i}(w)|(1-|w|^{2})^{n}\leq M_{i}\frac{n^{n}}{(n-1)!(N_{i}+n)}<\frac{1}{2^{i+2}}.$$
Notice that  since $M_i<i2^{i+1},$ one could have chosen $N_i=\frac{i2^{2i+3}n^n}{(n-1)!}.$  Furthermore, it can be shown that $N_i+2i<N_{i+1}.$ Thus, we have that
$$\frac{7}{8} < k_{w}(w)(1-|w|^{2})^{n} <\frac{9}{8},$$
and therefore, $k_{w}(w)(1-|w|^2)^n$ is indeed bounded by positive constants.

To show that $T$ is not an $n$-hypercontraction, we note the existence of some $n_{0}=N_{j}+j-1$ such that $${\frac{{w}_{n_{0}+1}}{{w}_{n_{0}}}}={\frac{j \widetilde{w}_{N_{j}+j}}{(j-1) \widetilde{w}_{N_{j}+j-1}}}={\frac{j}{j-1}}{\frac{N_{j}+j}{n+N_{j}+j-1}}>{\frac{1+n_0}{n+n_0}},$$ and apply Theorem \ref{main}.

Lastly, to show that $T$ and $S^*$ are not similar, we choose $n_0=N_{j}+j-1$ as in the previous case to get $$\frac{\prod\limits_{k=0}^{n_0}{\frac{\widetilde{w}_{k+1}}{\widetilde{w}_{k}}}}{\prod\limits_{k=0}^{n_0}{\frac{{w}_{k+1}}{{w}_{k}}}}
=\frac{\prod\limits_{k=0}^{N_{j}+j-1}{\frac{\widetilde{w}_{{k}+1}}{\widetilde{w}_{k}}}}{\prod\limits_{k=0}^{N_{j}+j-1}{\frac{{w}_{k+1}}{{w}_{k}}}}
={\frac{\widetilde{w}_{N_{j}+j}}{{w}_{N_{j}+j}}}={\frac{j{w}_{N_{j}+j}}{w_{N_{j}+j}}}={j}\rightarrow+\infty,$$
as $j\rightarrow+\infty$, and the conclusion follows from Lemma \ref{SH}.

\end{proof}

\section{  Trace of curvature and similarity of reducible operators in $B_m(\Omega)$. }

In this section, we give  a simple example to show that the $n$-hypercontraction assumption is needed to determine the similarity of operators in $B_m(\Omega)$ in terms of the trace of the curvatures as was claimed in \cite{HJK}. We first introduce some
definitions and mention some results about strongly irreducible operators. We assume throughout the section that $T \in \mathcal{L}(\mathcal{H})$.

\begin{Definition} $T$ is said to be strongly irreducible (denoted str-irred.) if there is no
nontrivial idempotent in the commutant ${\mathcal
A}^{\prime}(T)$ of $T$, that is, $T$ cannot be written as $$T=T_1\stackrel{\cdot}{+}T_2,$$ for some $T_i\in \mathcal{L}({\mathcal H}_i)$, where $1 \leq i \leq 2$, and
${\mathcal H}={\mathcal H}_1\stackrel{.}{+}{\mathcal H}_2.$

\end{Definition}

\begin{Definition}[\cite{CFJ}] Let $n < \infty$.
 A set ${\mathcal P}=\{P_i\}_{i=1}^{n}$ of idempotents in ${\mathcal L}({\mathcal H})$  is called a unit finite
decomposition of $T$ if

$1.\,\, P_i{\in} {\mathcal A}^{\prime}(T)$ for all $1{\leq}i{\leq}n$;

$2.\,\, P_iP_j={\delta}_{ij}P_i$  for all $1{\leq}i, j{\leq}n,$ where
${\delta}_{ij}=\left\{
\begin{array}{cc}
1 & i=j\\
0 & i \neq j
\end{array}
\right. $; and

$3. \,\, \sum\limits_{i=1}^{n}P_i=I_{\mathcal H},$ where $I_{\mathcal H}$
denotes the identity operator on $\mathcal H$.

If, in addition,

$4. \,\, P_i$ is a minimal idempotent in ${\mathcal A}^{\prime}(T)$, that is , $T|_{\ran P_i} \text{ is str-irred. for}$ $1 \leq i \leq n,$

then ${\mathcal P}$ is said to be a unit finite strong irreducible decomposition of $T$
and we call the cardinality of ${\mathcal P}$ the strong irreducible cardinality of
$T$.

\end{Definition}

It is clear that an operator $T$ has a unit
finite strong irreducible decomposition if and only if it can be expressed as
the direct sum of finitely many strongly irreducible operators.

\begin{Definition}[\cite{CFJ}]\label{UD}  Let ${\mathcal P}=\{P_i\}_{i=1}^{m}$ and ${\mathcal
Q}=\{Q_i\}_{i=1}^{n}$ be two unit finite strong irreducible decompositions of
$T$. We say that $T$ has a unique strong irreducible decomposition up to similarity
if $m=n$ and there exist an invertible operator $X \in {\mathcal
A}^{\prime}(T)$ and a permutation ${\Pi}$ of the set $(1, 2, \cdots,
n)$ such that $XQ_{{\Pi}(i)}X^{-1}=P_i$ for all $1{\leq}i{\leq}n.$

\end{Definition}

The work of the first author, in collaboration with X. Guo and C. Jiang, shows how this concept is related to Cowen-Douglas operators.

\begin{thm}[\cite {JGJ}] \label{JGJ} Let $T$ be a Cowen-Douglas operator and set $$T^{(n)}:=\bigoplus\limits^{n}_{i=1}T.$$ Then $ T^{(n)}$
has a unique strong irreducible decomposition up to similarity.
\end{thm}

Denote by $M_{k}({\mathcal A}^{\prime}(T))$ the collection of all $k\times k$ matrices with entries from the commutant ${\mathcal A}^{\prime}(T)$ of an operator $T$.
Let
 $$M_{\infty}({\mathcal A}^{\prime}(T))=\bigcup\limits^{\infty}_{k=1}M_{k}({\mathcal
 A}^{\prime}(T)),$$
and let $\mbox{\rm Proj}(M_k({\mathcal A}^{\prime}(T)))$ be the
algebraic equivalence classes of idempotents in
$M_{\infty}({\mathcal A}^{\prime}(T))$. Set $\bigvee ({\mathcal A}^{\prime}(T))=\mbox{\rm Proj}(M_{\infty}({\mathcal A}^{\prime}(T))).$

 If $p$ and $q$ are idempotents in $\bigvee ({\mathcal A}^{\prime}(T))$, then we will say that $p{\sim}_{st}q$ if  $p{\oplus}r$ and $q{\oplus}r$ are algebraically equivalent for some idempotent $r \in \bigvee ({\mathcal A}^{\prime}(T))$. The relation ${\sim}_{st}$ is known as stable equivalence.  The $K_0$-group of ${\mathcal A}^{\prime}(T)$, denoted  $K_0({\mathcal A}^{\prime}(T))$, is defined to be the Grothendieck group of $\bigvee ({\mathcal A}^{\prime}(T))$.

Now recall that for $\alpha \geq 1$, $\mathcal{M}_{\alpha}$ is a Hilbert space with reproducing kernel given by $K_{\alpha}(z, w)=\frac{1}{(1-\bar{w}z)^{\alpha}}$ and with the backward shift operator $S^*_{\alpha}$. Lemma \ref{SH} above shows that the backward shift operators on two different spaces cannot be similar.

 \begin{lm}\label{ab} $S^*_{\alpha}$ and $S^*_{\beta}$ are similar if and only if
$\alpha=\beta.$

\end{lm}

\begin{prop}\label{CFJ}Let $T=\bigoplus\limits_{k=1}^{l}S_{\alpha_k}^*$ and
$\widetilde{T}=\bigoplus\limits_{m=1}^{s}S_{\beta_m}^*$, where
$\alpha_k, \beta_m \geq 1$, and $\alpha_k\neq
\alpha_{k^{\prime}}$ and $\beta_{m}\neq \beta_{m^{\prime}}$, for $k \neq k'$ and $m \neq m'$. Then $T$ is
similar to $\widetilde{T}$ if and only if $l=s$ and there exists a
permutation ${\Pi}$ of the set $(1, 2, \cdots, l)$ such that for every $k\leq l$,
$\alpha_k=\beta_{\Pi(k)}$.
\end{prop}

\begin{proof} It suffices to prove one implication and therefore, we let $T$ and $\widetilde{T}$ be similar. Without loss of
generality, we assume that $l <  s$.  It is well-known that
$S^*_{\alpha_k}\in B_1(\mathbb{D})$ and that it is strongly irreducible
in $\mathcal{L}({\mathcal M}_{\alpha_k})$. If we let ${\mathcal
H}=\bigoplus\limits_{k=1}^l{\mathcal M}_{\alpha_k}$, then  $T|_{Ran I_{{\mathcal H}_k}}=S^*_{\alpha_k}$. Analogous results hold for the operator $\widetilde{T}$. Moreover,
since $T\in B_l(\mathbb{D})$ and $\widetilde{T}\in B_s(\mathbb{D})$, $$T\oplus \widetilde{T}= \bigoplus\limits_{k=1}^lS^*_{\alpha_k}\oplus
\bigoplus\limits_{m=1}^{s}S_{\beta_m}^* \in
B_{l+s}(\mathbb{D}).$$

By Theorem \ref{JGJ}, for any positive integer $n$, both $T^{(n)}\oplus \widetilde{T}^{(n)}$ and $T^{(2n)}$ 
 have unique strong irreducible decompositions up to similarity. Moreover,
we know from Lemma \ref{ab} that $S^*_{\alpha_k}$ and $S^*_{\alpha_k'}$ are not similar for $k \neq k'$. The same is true for $S^*_{\beta_m}$.

Now, the results in \cite{CFJ} show that
$$ K_0 \left ({\mathcal A}^{\prime} \left (\bigoplus\limits_{j=1}^t T_j \right ) \right) \cong
\mathbb{Z}^{t},$$
 where the $T_j$ are strongly irreducible Cowen-Douglas operators such that no two of them are similar to each other.
Since each $S^*_{\alpha_k}$ is a strongly irreducible Cowen-Douglas operator, we then have
$$K_0({\mathcal A}^{\prime}(T^{(2)}))=K_0(M_{2}(\mathbb{C})\otimes {\mathcal
A}^{\prime}(T))= K_0({\mathcal A}^{\prime}(T)) \cong
\mathbb{Z}^{l}.$$ Notice that  if $T\oplus \widetilde{T}$ is similar to
$T^{(2)}$, then
$$K_0({\mathcal A}^{\prime}(T\oplus \widetilde{T}))=K_0({\mathcal A}^{\prime}(T^{(2)})) \cong
\mathbb{Z}^{l}.$$
On the other hand,  if there exists a $\beta_m$
such that $S^*_{\beta_m}$ is not similar to any $S^*_{\alpha_k}$ in $T\oplus \widetilde{T}=
\bigoplus\limits_{k=1}^lS^*_{\alpha_k}\oplus
\bigoplus\limits_{m=1}^{s}S_{\beta_m}^*,$ then one can find a positive number
$l^{\prime}>l$ such that
$$K_0({\mathcal A}^{\prime}(T\oplus \widetilde{T})) \cong
\mathbb{Z}^{l^{\prime}}.$$ This is a contradiction.

\end{proof}

The following example shows that the $n$-hypercontractivity assumption cannot be dispensed with in determining similarity. A more general example can be constructed in the same way.

\begin{ex}

For every $w \in \mathbb{D}$, $$\text{trace }\mathcal{K}_{S^*_1\oplus S^*_3}(w)=-\frac{2}{(1-|w|^2)^2}=\text{trace }\mathcal{K}_{S^*_2\oplus S^*_2}(w).$$
But by Proposition \ref{CFJ}, we know that $S^*_1\oplus S^*_3$ is similar to
$S^*_2\bigoplus S^*_2$ if and only if both $S^*_1$ and $S^*_3$ are similar to $S^*_2$, which is a contradiction. In fact, since
$S^*_1$ is not a 2-hypercontraction, $S^*_1 \oplus S^*_3$ cannot be a 2-hypercontraction, either.

\end{ex}

\end{document}